\documentclass[a4paper,12pt]{amsart}
\usepackage[utf8]{inputenc}
\usepackage[T1]{fontenc}
\usepackage{mathptmx}
\usepackage{amsfonts}

\usepackage[english]{babel}
\usepackage{amssymb,amsmath}
\usepackage{xcolor}

\newtheorem{Theo}{Theorem}[section]
\newtheorem{Prop}[Theo]{Proposition}
\newtheorem{Coro}[Theo]{Corollary}
\newtheorem{Lemm}[Theo]{Lemma}
\newtheorem{Defi}[Theo]{Definition}

\newtheorem{Rema}[Theo]{Remark}

\newcommand{\diam}{\operatorname{diam}}
\newcommand{\conv}{\operatorname{co}}
\newcommand{\aconv}{\operatorname{aco}}

\newcommand{\spn}{\operatorname{span}}

\begin{document}

\title{A Bourgain-like property of Banach spaces \linebreak with no copies of $c_0$}

\author{A. P\'erez}
\address{Departamento de Matem\'{a}ticas, Universidad de Murcia,
30100 Espinardo (Murcia), Spain} \email{antonio.perez7@um.es}

\author{M. Raja}
\address{Departamento de Matem\'{a}ticas, Universidad de Murcia,
30100 Espinardo (Murcia), Spain} \email{matias@um.es}

\thanks{A. P\'erez and M. Raja are partially supported by the  
MINECO/FEDER project  {MTM2014-57838-C2-1-P}. The first author is also supported by a PhD fellowship of La Caixa Foundation.}

\maketitle

%%%%%%%%%%%%%%%%%%%%%%%%%%%%%%%%%%%%%%%%%%%%%%%%%%%%%%%%%%%%%%%%%%%%%%%%%%%%%%%%%%%%%%%%%%%%%%%

\begin{abstract}
We give a characterization of the existence of copies of $c_{0}$ in Banach spaces in terms of indexes. As an application, we deduce new proofs of James Distortion theorem and Bessaga-Pe\l{}czynski theorem about weakly unconditionally Cauchy series.
\end{abstract}

\section{Introduction}

The aim of this paper is to study the existence of copies of $c_{0}$ in Banach spaces in terms of indexes and by purely geometrical methods. Our motivation for this is the beautiful characterization given by Bourgain \cite[Lemma 3.7, p. 39]{BourgainRNP} of Banach spaces not containing $\ell^{1}$, as those satisfying that for every bounded subset $C$ of $X^{\ast}$ and each $\epsilon > 0$ there exist relatively weak$^{\ast}$-open subsets $U_{1}, \ldots, U_{m}$ of $C$ such that $\frac{1}{m}(U_{1} + \ldots + U_{m})$ has diameter less than $\epsilon$. Several results concerning this type of spaces follow from this, like the fact that their dual unit ball $(B_{X^{\ast}}, \omega^{\ast})$ is convex block compact \cite[Proposition 3.11, p. 43]{BourgainRNP}.

We prove that a Banach space $X$ does not contain an isomorphic copy of $c_{0}$ if and only if for every bounded subset $A$ of $X$ and each $\epsilon > 0$ there are $x_{1}, \ldots, x_{m}$ in $A$ such that $\bigcap_{j=1}^{m}{(A - x_{j}) \cap (x_{j} - A)}$ has diameter less than $\epsilon$. Actually, we give a quantitative version of this fact. We first associate to any bounded set $A \subset X$ a sequence of indexes $\delta_{m}(A)$ ($m \geq 0$), being $\delta_{m}(A)$ half of the infimum of all diameters of sets $\bigcap_{j=1}^{m}{(A - x_{j}) \cap (x_{j} - A)}$ where $x_{1}, \ldots, x_{m} \in A$. Then, we prove in Theorem \ref{theo:c0Copies} that for each $\epsilon > 0$ we can find a sequence $(x_{n})_{n \in \mathbb{N}}$ in the absolute convex hull of $A$ such that
\begin{equation} 
\label{equa:equivalentBasis}
\left( \delta_{2^{N}}(A) - \epsilon \right) \cdot \max_{1 \leq n \leq N}{|\lambda_{n}|} \leq \left\| \sum_{n=1}^{N}{\lambda_{n} x_{n}} \right\| \leq  \delta_{0}(A) \cdot \max_{1 \leq n \leq N}{|\lambda_{n}|} 
\end{equation}
for every $\lambda_{1}, \ldots \lambda_{N} \in \mathbb{R}$ and $N \in \mathbb{N}$. 

From the previous result we deduce the characterization of Banach spaces containing an isomorphic copy of $c_{0}$ mentioned above (Theorem \ref{theo:c0chacterization}), as well as the known theorems of James (Theorem \ref{theo:JamesDistortion}) and Bessaga-Pe\l{}czynski (Theorem \ref{theo:BessagaPelczynski}) without using basic sequences. 

Our notation is standard and follows \cite{Fabian}. We denote by $X$ a real Banach space with the norm $\| \cdot \|$. Its topological dual will be denoted by $X^{\ast}$, and for any $x^{\ast} \in X^{\ast}$ and $x \in X$ the evaluation of $x^{\ast}$ at $x$ is written as $x^{\ast}(x) = \langle x^{\ast}, x \rangle = \langle x, x^{\ast} \rangle$. The closed unit ball (resp. unit sphere) of $X$ is denoted by $B_{X}$ (resp. $S_{X}$). If $D \subset X$ then we write $\conv{(D)}$, $\aconv{(D)}$ and $\spn{(D)}$ to denote the convex hull, the absolutely convex hull and the linear hull of $D$. The supremum of $x^{\ast} \in X$ on $D$ is denoted by $\sup{(x^{\ast}, D)}$. A \emph{slice} of $D$ is a set of the form $S(D, x^{\ast}, \delta):=\{ x \in D: x^{\ast}(x) > \sup{(x^{\ast}, D)} - \delta \}$ for some $x^{\ast} \in X^{\ast}$ and $\delta > 0$. Recall that the \emph{diameter} of $D$ is defined as
$\diam{(D)} := \sup{\{ \| x- y\|\colon x,y \in D \}}$. 

\section{Indexes of symmetrization}

\begin{Defi}
Given $A \subset X$ bounded, the symmetrized of $A$ with respect to $x_{1}, \ldots, x_{N} \in A$ is defined as
$\bigcap_{n=1}^{N}{(A - x_{n})\cap (x_{n} - A)}$. 
For each $N \in \mathbb{N}$ we also define:
\[
\Delta_{N}(A) := \left\{ \mbox{$\bigcap_{n=1}^{N}{(A - x_{n})\cap (x_{n} - A)}$}\colon x_{1}, \ldots, x_{N} \in A \right\}. 
\]
\[
 \delta_{0}(A) := \diam(A)/2, \hspace{3mm} \delta_{N}(A) := \inf{\left\{ \delta_{0}(D): D \in \Delta_{N}(A) \right\}}. 
 \]
 \end{Defi}
It is clear from the definition that $\{\Delta_{N}(A)\colon N \in \mathbb{N}\}$ is an increasing sequence of sets, and hence $\{ \delta_{N}(A)\colon N \in \mathbb{N} \}$ is decreasing. We will write $\delta_{\infty}(A) := \lim_{N}{\delta_{N}(A)}$. Let us point out that if $x \in A$, then $d \in (A- x) \cap (x - A)$ is equivalent to $x \pm d \in A$. With this in mind, the following (useful) observations are direct:
\begin{enumerate}
\item[(I)] If $D \in \Delta_{N}(A)$ is the symmetrized of $A$ with respect to $x_{1}, \ldots, x_{N} \in A$, then for every $d \in D$ the set $(D- d) \cap (d -D)$ is the symmetrized of $A$ with respect to $x_{1} \pm d, \ldots, x_{N} \pm d \in A$. In particular, $(D- d) \cap (d -D) \in \Delta_{2N}(A)$.
\item[(II)] Given $x^{\ast} \in X^{\ast}$, $\delta > 0$ and $x \in S(A, x^{\ast}, \delta)$, every $d \in D:=(A - x) \cap (x- A)$ satisfies $|x^{\ast}(x)| + |x^{\ast}(d)| < \sup(x^{\ast}, A)$, so that $|x^{\ast}(d)| < \delta$. In particular, $ x \pm D \subset S(A,x^{\ast},2\delta)$.
\end{enumerate}
Recall that the \emph{Kuratowski measure of non-compactness} of a set $S \subset X$ is
\[ \alpha(S) := \inf{\{ \epsilon > 0: \hbox{ there are finitely many balls of radius $\epsilon$ which cover $S$} \}}. \]

\begin{Lemm}
\label{LEMA:KuratowskiSizeSym}
If $A \subset X$ is bounded and $D \in \Delta_{N}(A)$, then $\alpha(D) \geq \delta_{2N}(A)$. 
\end{Lemm}

\begin{proof}
Suppose that $\alpha(D) < \epsilon$, and let $D_{1}, \ldots, D_{n}$ be a finite family of subsets of $D$ whose union is equal to $D$ and such that each $D_{k}$ is contained in a ball of radius less than $\varepsilon$. If $D \subset \overline{\conv}{(D_{1})}$, then $\diam(D) < 2 \epsilon$ and so $\delta_{2N}(A) \leq \delta_{N}(A) \leq \delta_{0}(D) < \epsilon$. Otherwise, we can assume that there is $ 2 \leq m \leq n$ such that
\begin{equation} 
\label{equa:convexHullSeparation}
D \subset \overline{\conv}{\left( D_{1} \cup \ldots \cup D_{m}\right) } \mbox{ \hspace{3mm} and \hspace{3mm} } D \nsubseteq \overline{\conv}{\left( D_{1} \cup \ldots \cup D_{m-1}\right) }. 
\end{equation}
We can take $x_{0}^{\ast} \in S_{X^{\ast}}$ and $\delta > 0$ such that the slice $S(D,x_{0}^\ast, \delta)$ has empty intersection with $\overline{\conv}{\left( B_{1} \cup \ldots \cup B_{m-1}\right) }$. We claim that for every $0 < \eta < 1$ it holds that 
\begin{equation}
\label{equa:claimSlice}
S(D, x_{0}^{\ast}, \eta \delta) \subset D_{m} + \eta (1 + \diam{D}) B_{X}. 
\end{equation}
If $d \in S(D, x_{0}^\ast, \eta \delta)$, then by \eqref{equa:convexHullSeparation} we can find $d':=\lambda d_{m} + (1 - \lambda) c_{m}$ where $0 \leq \lambda \leq 1$, $d_{m} \in D_{m}$ and $c_{m} \in \conv{(D_{1} \cup \ldots \cup D_{m-1})}$ such that $\| d - d' \| < \eta$ and $d' \in S(D, x_{0}^{\ast}, \delta \eta)$. Since $x_{0}^{\ast}(c_{m}) \leq \sup{(x_{0}^{\ast})} - \delta$, we deduce that
\[ \sup(x_{0}^\ast, D) - \eta \delta < x_{0}^{\ast}(d') = \lambda x_{0}^{\ast}(d_{m}) + (1 - \lambda) x_{0}^{\ast}(c_{m}) \leq \sup{(x_{0}^\ast, D)} - (1 - \lambda) \delta. \]
This yields $1 - \lambda < \eta$, and so 
\[ \| d - d_{m}\| \leq \| d - d' \| + \| d' - d_{m} \| < \eta + (1 - \lambda) \| d_{m} - c_{m} \| < \eta (1 + \diam{D}). \]
This proves the claim. By observations (I) and (II), for every $d_{0} \in S(D,x_{0}^{\ast}, \eta \delta/2)$ the set $D_{0} := (D - d_{0}) \cap (d_{0} - D)$ belongs to $\Delta_{2N}(A)$ and $d_{0} \pm D_{0} \subset S(D,x_{0}^{\ast}, \eta \delta)$. Hence, we get by \eqref{equa:claimSlice} that
\[ \delta_{2N}{(A)} \leq \delta_{0}{(D_{0})} \leq \frac{1}{2} \diam{S(D,x_{0}^{\ast}, \eta \delta)} \leq \varepsilon + \eta(1 + \diam{D}). \] 
Since $\eta > 0$ is arbitrary, we conclude that $\delta_{2N}{(A)} < \varepsilon$.
\end{proof}

\begin{Rema}
We are thankful to an anonymous referee for pointing out to us that Lemma \ref{LEMA:KuratowskiSizeSym} can be obtained as a corollary of the so-called ``Superlemma'' of Namioka and Bourgain \cite[Chapter IX, p. 157]{DiestelSeq}. Indeed, under the assumption \eqref{equa:convexHullSeparation} we can apply this result to the closed convex hull of $D$ to obtain a slice $S=S(D, x_{0}^{\ast}, \delta)$ of $D$ with diameter smaller than the diameter of $\overline{\conv}{(D_{m})}$, which is less than $2 \varepsilon$. Taking $d_{0} \in S(D, x_{0}^{\ast}, \delta/2)$ we can argue as in the last part of the proof of Lemma \ref{LEMA:KuratowskiSizeSym} to conclude the result. 
\end{Rema}

%%%%%%%%%%%%%%%%%%%%%%%%%%%%%%%%

\begin{Lemm}
\label{LEMA:orthogonal}
Let $F \subset X$ be a finite-dimensional subspace and $D \subset X$ bounded. If $\alpha(D) > \lambda > 0$, then there exists $x_{0}^{\ast} \in S_{F^{\perp}}$ such that $\sup{(x_{0}^{\ast}, D)} > \lambda$.
\end{Lemm}

\begin{proof}
Suppose that every $x_{0}^{\ast} \in S_{F^\perp}$ satisfies that $\sup{(x_{0}^{\ast},D)} \leq \lambda$.  By Hahn-Banach Theorem we have that $D \subset F + \lambda B_{X}$. But then $D \subset \mu B_{F} + \lambda B_{X}$ for some $\mu > 0$, which implies that $\alpha(D) \leq \lambda$ by the compactness of $B_{F}$.
\end{proof}

\begin{Theo}
\label{theo:c0Copies}
Let $A \subset X$ be bounded. For every $\epsilon > 0$ there is a sequence $(x_{n})_{n \in \mathbb{N}}$ in $\aconv{(A)}$ such that
\begin{equation}\label{equa:c0copies}  
(\delta_{2^{N}}(A) - \epsilon) \cdot \max_{1 \leq n \leq N}{|\lambda_{n}|} \leq \left\| \sum_{n=1}^{N}{\lambda_{n} x_{n}} \right\| \leq  \delta_{0}(A) \cdot \max_{1 \leq n \leq N}{|\lambda_{n}|} 
\end{equation}
for every $\lambda_{1}, \ldots, \lambda_{N}$ in $\mathbb{R}$ and $N \in \mathbb{N}$.
\end{Theo}

\begin{proof}
Write $\eta = \epsilon/3$. Fix $x_{0} \in A_{0}:=A$ and put $A_{1}:= (A -x_{0}) \cap (x_{0} - A)$. By Lemma \ref{LEMA:KuratowskiSizeSym} we have that $\alpha(A_{1}) > \delta_{2}(A) - \eta$, so Lemma \ref{LEMA:orthogonal} yields that there are $x_{1} \in A_{1}$ and $x_{1}^{\ast} \in S_{X^{\ast}}$ with 
$x_{1}^{\ast}(x_{1}) >  \sup{(x_{1}^{\ast}, A_{1})} - \eta > \delta_{2}(A) - 2\eta$. Suppose that $N \geq 1$ and we have constructed $(x_{n}^{\ast})_{n=1}^{N}$ in $S_{X^{\ast}}$, $(x_{n})_{n=1}^{N}$ in $\aconv{(A)}$ and $(A_{n})_{n=1}^{N}$ subsets of $X$ satisfying for each $1 \leq n \leq N$:
\begin{itemize}
\item[(a)] $x_{n-1} \pm A_{n} \subset A_{n-1}$ and $A_{n} \in \Delta_{2^{n-1}}(A)$.
\item[(b)] $\{ x_{k}: 1 \leq k < n \} \subset \ker{x_{n}^{\ast}}$.
\item[(c)] $x_{n}^{\ast}(x_{n})  >  \sup{(x_{n}^{\ast}, A_{n})} - \eta > \delta_{2^{n}}(A) - 2\eta$.
\end{itemize}
Put $A_{N+1} := (A_{N} - x_{N}) \cap (x_{N} - A_{N}) \in \Delta_{2^N}(A)$. By Lemma \ref{LEMA:KuratowskiSizeSym} we have that $\alpha(A_{N+1}) > \delta_{2^{N+1}}(A) - \eta_{N+1}$, so using Lemma \ref{LEMA:orthogonal} we obtain $x_{N+1} \in A_{N+1}$ and $x_{N+1}^{\ast} \in S_{X^{\ast}}$ such that $\{ x_{k}: 1 \leq k \leq N \} \subset \ker{x_{N+1}^{\ast}}$ and $x_{N+1}^{\ast}(x_{N+1}) > \sup{(x_{N+1}^{\ast}, A_{N+1})} - \eta > \delta_{2^{N+1}}(A) - 2\eta$. This finishes the inductive construction. Notice that conditions (a) and (c) imply that
\begin{enumerate}
\item[(d)] $|x_{n}^{\ast}(z)| < \eta$ whenever $z \in A_{n+1}$.
\end{enumerate}
Given $N \in \mathbb{N}$, we show now that the sequence $(x_{n})_{n \in \mathbb{N}}$ satisfies \eqref{equa:c0copies}. For every $0 \neq (\lambda_{n})_{n=1}^{N} \in \mathbb{R}^{N}$ we can write
\[ \left\| \sum_{n=1}^{N}{\lambda_{n} x_{n}} \right\| = |\lambda_{m}| \cdot \left\| \sum_{n=1}^{N}{\frac{\lambda_{n}}{\lambda_{m}} x_{n}} \right\| \leq |\lambda_{m}| \cdot \delta_{0}(A) \]
being $m$ such that $|\lambda_{m}| = \max{\{ |\lambda_{n}|: 1 \leq n \leq N \}}$, since $x_{0} + \sum_{n=1}^{N}{\pm x_{n}} \in A$. Furthermore 
\[
\left\| \sum_{n=1}^{N}{\frac{\lambda_{n}}{\lambda_{m}} x_{n}} \right\| \geq \langle x_{m}^{\ast}, x_{m} \rangle + \langle x_{m}^{\ast}, \sum_{m < n \leq N}{\frac{\lambda_{n}}{\lambda_{m}} x_{n}} \rangle \geq \delta_{2^{m}}(A) - 3\eta \geq \delta_{2^{N}}(A) - 3 \eta,
\]
where we have used (b), (c), (d) and the fact that
\[ \sum_{m < n \leq N}{\frac{\lambda_{n}}{\lambda_{m}} x_{n}} \in \conv{(A_{m+1})}, \mbox{ which is a consequence of (a)}. \]
\end{proof}

\begin{Coro}
\label{coro:c0copy}
Let $A \subset X$ be bounded. For every $\epsilon > 0$ there is a sequence in $(x_{n})_{n \in \mathbb{N}}$ in $\aconv{(A)}$ such that
\[ (\delta_{\infty}(A) - \epsilon) \max_{n \in \mathbb{N}}{|\lambda_{n}|} \leq \left\| \sum_{n=1}^{\infty}{\lambda_{n} x_{n}} \right\| \leq \delta_{0}(A) \cdot \max_{n \in \mathbb{N}}{|\lambda_{n}|} \]
for every finitely supported sequence $(\lambda_{n})_{n \in \mathbb{N}}$ in $\mathbb{R}$.
\end{Coro}

\section{Copies of $c_{0}$ in Banach spaces}

\begin{Theo}
\label{theo:c0chacterization}
Let $X$ be a Banach space. The following assertions are equivalent:
\begin{enumerate}
\item[(i)] $c_{0}$ is not isomorphic to a subspace of $X$.
\item[(ii)] $\delta_{\infty}(C) = 0$ for every bounded set $C \subset X$.  
\item[(iii)]  $\delta_{\infty}(C) = 0$ for every bounded, convex and closed set $C \subset X$.
\end{enumerate}
\end{Theo}

\begin{proof}
Implication (i)$\Rightarrow$(ii) is a consequence of Corollary \ref{coro:c0copy}, while (ii)$\Rightarrow$(iii) is obvious. We just have to check that (iii)$\Rightarrow$(i). Let $T: c_{0} \rightarrow X$ be an isomorphism, and consider $A := T(B_{c_{0}})$. Given $a_{1}, \ldots, a_{N} \in A$ and $0 < \epsilon < 1$ we can find $m \in \mathbb{N}$ such that $a_{n} \pm (1 - \epsilon) T(e_{m}) \in A$ for every $1 \leq n \leq N$. This shows that $\delta_{N}(A) \geq (1 - \epsilon)/\| T^{-1} \|$ for each $N \in \mathbb{N}$.
\end{proof}

When $c_{0}$ is isomorphic to a subspace of $X$, it is also said that $X$ has a copy of $c_{0}$. It turns out that these spaces have indeed almost isometric copies of $c_{0}$, which means that for every $\epsilon > 0$ we can find a closed subspace $Y \subset X$ and an isomorphism $T: c_{0} \rightarrow Y$ such that $\| T \| \| T^{-1}\| \leq 1+ \epsilon$.

\begin{Theo}[James]
\label{theo:JamesDistortion}
If $X$ has a copy of $c_{0}$, then it has almost isometric copies of $c_{0}$.
\end{Theo}

\begin{proof}
If $c_{0}$ embedds into $X$ then there exists a bounded set $A \subset X$ with $\delta_{\infty}(A) > 0$ by Theorem \ref{theo:c0chacterization}. It follows from the definition of $\delta_{\infty}(A)$ that for every $\epsilon > 0$ there exists an element $D \in \bigcup_{N \in \mathbb{N}}{\Delta_{N}(A)}$ such that 
\[ \delta_{\infty}(A) \leq \delta_{0}(D) \leq (1 + \epsilon)\delta_{\infty}(A). \]
Since $\delta_{\infty}(A) \leq \delta_{\infty}(D)$, we deduce that $\delta_{0}(D) \leq (1 + \epsilon) \delta_{\infty}(D)$, so an application of Corollary \ref{coro:c0copy} with $D$ leads to the desired copy of $c_{0}$.
\end{proof}

Another easy consequence is the Bessaga-Pe\l{}zcynski criterion for the existence of copies of $c_{0}$. Recall that a series $\sum_{n}{x_{n}}$ in a Banach space $X$ is said to be wuC if $\sum_{n}{|x^{\ast}(x_{n})|}$ converges for every $x^{\ast} \in X^{\ast}$, which by the Uniform Boundedness Principle implies that $\sum_{n}{|x^{\ast}(x_{n})|}$ is uniformly bounded for $x^{\ast} \in B_{X^{\ast}}$.

\begin{Theo}[Bessaga-Pe\l{}czynski]
\label{theo:BessagaPelczynski}
If $c_{0} \nsubseteq X$ and $\sum_{n}{x_{n}}$ is wuC, then the series is unconditionally convergent.
\end{Theo}

\begin{proof}
Consider the uniformly bounded sets given by
\[ A_{m}= \left\{ \sum_{n=1}^{m}{\theta_{n} x_{n}}: \theta_{n} \in \{ -1,1\} \mbox{ for each } 1 \leq n \leq m \right\}, \hspace{4mm} A = \bigcup_{m \in \mathbb{N}}{A_{m}}. \]
If $X$ does not contain a copy of $c_{0}$, then $\delta_{\infty}(A) =0$, so given $\epsilon > 0$ we can find $a_{1}, \ldots, a_{N} \in A$ with
\[ \diam\left( \bigcap_{j=1}^{N}{(A - a_{j}) \cap (a_{j} - A)} \right) < \epsilon. \]
There is $M \in \mathbb{N}$ such that $\{ a_{n}: 1 \leq n \leq N\} \subset \bigcup_{m \leq M}{A_{m}}$, so $\left\| \sum_{n=M}^{M'}{\theta_{n} x_{n}} \right\| \leq \epsilon$ for every $\theta_{n} \in \{ -1,1\}$ and $M' \geq M$.
\end{proof}

We finish with a non-symmetrized characterization of Banach spaces with no copies of $c_{0}$.

\begin{Prop}
A Banach space $X$ does not contain an isomorphic copy of $c_{0}$ if and only if for every bounded set $A \subset X$ and each $\epsilon > 0$ there are $x_{1}, \ldots, x_{N} \in A$ such that
\[ \diam\left(\bigcap_{j=1}^{N}{(A - x_{j})}\right) < \epsilon. \]
\end{Prop}

\begin{proof}
The sufficiency of the condition is consequence of Theorem \ref{theo:c0chacterization}. To see the converse, assume that there exists $A \subset X$ and $\epsilon > 0$ such that any intersection like in the statement has diameter greater or equal than $\epsilon$. Fix an arbitrary $x_{0} \in A$ and then pick $x_{1} \in (A - x_{0})$ such that $\| x_{1}\| \geq \varepsilon$. Consider the set $A_{1} := \{ x_{0}, x_{0} +x_{1}\} \subset A$. Now we take
\[ x_{2} \in \bigcap_{x \in A_{1}}{(A - x)} \, \mbox{ with $\| x_{2} \| \geq \varepsilon$} \mbox{ \hspace{3mm} and \hspace{3mm} } A_{2} := A_{1} \cup (A_{1} + x_{2}). \]
Following in this way, we will have a sequence $(x_{n})_{n \in \mathbb{N}}$ of vectors of norm greater or equal to $\epsilon$ for $n \geq 1$ and sets $A_{n} \subset A$ of cardinality $2^{n}$. Then consider
\[ x_{n+1} \in \bigcap_{x \in A_{n}}{(A - x)} \, \mbox{ with $\| x_{n+1} \| \geq \varepsilon$} \mbox{ \hspace{3mm} and \hspace{3mm} } A_{n+1}:=A_{n} \cup (x_{n} + A_{n}). \] 
Notice that the sums $\sum_{n=1}^{N}{\theta_{n} x_{n}}$ are uniformly bounded independently of $N$ and the choice of $\theta_{n} \in \{ -1,1\}$, since they are difference of two elements of $A_{N} \subset A$. Now Theorem \ref{theo:BessagaPelczynski} implies that $X$ contains a copy of $c_{0}$.
\end{proof}

\section{Remarks}

Let $A$ be a subset of $X$. Recall that an $\epsilon$-tree in $A$ is a a sequence $\{ x_{n}\colon n \in \mathbb{N}\}$ such that $x_{n} = (x_{2n} + x_{2n+1})/2$ and $\| x_{2n} - x_{2n+1}\| \geq \epsilon$ for every $n \in \mathbb{N}$. The index $\delta_{1}(A)$ is directly related to existence of $\epsilon$-trees inside $A$. In fact, if $\delta_{1}(A) > \epsilon$, then we can construct a $2\epsilon$-tree inside of $A$ in the following way: fix any $x_{1} \in A$. Since $\diam((A - x_{1}) \cap (x_{1} - A)) > 2 \epsilon$, we can find $u_{1} \in X$ such that $\| u_{1}\| \geq \varepsilon$ and $x_{1} \pm u_{1} \in A$. Put $x_{2} := x_{1} - u_{1}$ and $x_{3}:=x_{1} + u_{1}$. Repeating this process with $x_{2}, x_{3}$ and the subsequent constructed points, we obtain the desired $2 \epsilon$-tree. On the other hand, it is clear that every $\epsilon$-tree $A'$ satisfies that $\delta_{1}(A') \geq \epsilon/2$. As a consequence, we can conclude that a set $A \subset X$ contains no $\epsilon$-trees (for any $\epsilon > 0$) if and only if $\delta_{1}(A') = 0$ for each $A' \subset A$. In particular, if $C$ is a closed and convex set having the Radon-Nikod\'{y}m Property (RNP), then $\delta_{1}(A) = 0$ for every $A \subset C$.

We say that $x_{0} \in A$ is an \emph{$\epsilon$-extreme point} of $A$ if $\diam((A - x_{0}) \cap (x_{0} - A))$ is less than $2\epsilon$. It is not difficult to see that $x_{0}$ is an extreme point of $A$ if and only if it is $\epsilon$-extreme for every $\epsilon > 0$. As a consequence, if $K \subset X$ is a bounded, closed and convex set having the Krein-Milman Property (KMP), then $\delta_{1}(C) = 0$ for every closed and convex set $C \subset K$. 

The previous notion reminds of the following concept introduced by Kunen and Rosenthal \cite{MartingaleProofs}: $x_{0} \in A$ is an \emph{$\epsilon$-strong extreme point} of $A$ if there is $\delta> 0$ such that whenever $a_{1}, a_{2} \in A$ and there exists a point $u=\lambda a_{1} + (1 - \lambda) a_{2}$ ($0 < \lambda < 1$) with $\| x_{0} - u\| < \delta$, then $\| u - a_{1}\| < \epsilon$ or $\|u - a_{2}\| < \epsilon$. If $x_{0}$ is $\epsilon$-strong extreme for every $\epsilon > 0$, then we simply say that it is a \emph{strong extreme point}. It is not difficult to see that every $\epsilon$-strong extreme point of $A$ is an $\epsilon$-extreme point of the same set. The converse is not true, since as it is pointed out in \cite[Remark 3, p. 173]{MartingaleProofs} every strong extreme point of a bounded, closed and convex set is also an extreme point of its $\sigma(X^{\ast \ast}, X^{\ast})$-closure (in the terminology of \cite{PreservedPoints} we might say that these are \emph{preserved extreme points}), while there are, for instance, Banach spaces where $B_{X}$ has extreme points that are not extreme points of $B_{X^{\ast \ast}}$ (see \cite{PreservedPoints}). With this formulation we have that if $K$ is a bounded, closed and convex set such that every $A \subset K$ has $\epsilon$-extreme points for every $\epsilon > 0$ (i.e. $\delta_{1}(A) = 0$), then each closed and convex set $C \subset K$ has an $\epsilon$-strong extreme point for every $\epsilon > 0$ (see \cite[Proposition 3.2, p. 170]{MartingaleProofs}).

%%%%%%%%%%%%%%%%%%%%%%%%%%%%%%%%%%%%%%%%%%%%%%%%%%%%%%%%%%%%%%%%%%%%%%%%%%%%%%%%%%%%%%%%%%%%%%%

\end{document}